\documentclass[11pt,twoside]{article}
\usepackage[utf8]{inputenc}
\usepackage{amsmath}
\usepackage{amsfonts}
\usepackage{amssymb}
\usepackage{amsthm}
\usepackage[english]{babel}
\usepackage{hyperref}
\usepackage[alphabetic]{amsrefs}
\usepackage{geometry}
\usepackage{graphicx}
\usepackage{mathrsfs}
\usepackage{parskip}
\usepackage{paralist}

\hypersetup{
	colorlinks   = true, 
	urlcolor     = blue, 
	linkcolor    = blue,
	citecolor    = blue 
}

\usepackage{fancyhdr}
\makeatletter
\def\thm@space@setup{%
	\thm@preskip=\parskip \thm@postskip=0pt
}
\makeatother
\usepackage{etoolbox}
\AtBeginEnvironment{proof}{\vspace*{-\parskip}}

\usepackage{titlesec}
\titlespacing{\section}{0pt}{1ex}{-1ex}
\titlespacing{\subsection}{0pt}{1ex}{-1ex}

\begin{document}
\title{\vspace{-1cm}\textbf{A sub-additive inequality for the volume spectrum}}
\author{Akashdeep Dey \thanks{Email: adey@math.princeton.edu, dey.akash01@gmail.com}}
\date{}
\maketitle

\theoremstyle{plain}
\newtheorem{thm}{Theorem}[section]
\newtheorem{lem}[thm]{Lemma}
\newtheorem{pro}[thm]{Proposition}
\newtheorem{clm}[thm]{Claim}
\newtheorem*{thm*}{Theorem}
\newtheorem*{lem*}{Lemma}
\newtheorem*{clm*}{Claim}

\theoremstyle{definition}
\newtheorem{defn}[thm]{Definition}
\newtheorem{ex}[thm]{Example}
\newtheorem{rmk}[thm]{Remark}

\numberwithin{equation}{section}

\newcommand{\mf}{manifold\;}
\newcommand{\vf}{varifold\;}
\newcommand{\hy}{hypersurface\;}
\newcommand{\Rm}{Riemannian\;}
\newcommand{\cn}{constant\;}
\newcommand{\mt}{metric\;} 
\newcommand{\st}{such that\;}
\newcommand{\Thm}{Theorem\;}
\newcommand{\Lem}{Lemma\;}
\newcommand{\Pro}{Proposition\;}
\newcommand{\eqn}{equation\;}
\newcommand{\te}{there exist\;}\newcommand{\tes}{there exists\;}\newcommand{\Te}{There exist\;}\newcommand{\Tes}{There exists\;}
\newcommand{\tf}{Therefore,\;} \newcommand{\hn}{Hence,\;}\newcommand{\Sn}{Since\;}\newcommand{\sn}{since\;}\newcommand{\nx}{\Next,\;}\newcommand{\df}{define\;}
\newcommand{\wrt}{with respect to\;}
\newcommand{\bbr}{\mathbb{R}}\newcommand{\bbq}{\mathbb{Q}}
\newcommand{\bbn}{\mathbb{N}}
\newcommand{\bbz}{\mathbb{Z}}
\newcommand{\mres}{\scalebox{1.8}{$\llcorner$}}
\newcommand{\ra}{\rightarrow}
\newcommand{\fn}{function\;}
\newcommand{\lra}{\longrightarrow}
\newcommand{\sps}{Suppose\;}
\newcommand{\del}{\partial}
\newcommand{\seq}{sequence\;}
\newcommand{\cts}{continuous\;} 
\newcommand{\bF}{\mathbf{F}} 
\newcommand{\bM}{\mathbf{M}} 
\newcommand{\bL}{\mathbf{L}}
\newcommand{\cm}{\mathcal{C}(M)}
\newcommand{\zn}{\mathcal{Z}_n(M; \mathbb{Z}_2)}

\newcommand{\cA}{\mathcal{A}}\newcommand{\cB}{\mathcal{B}}\newcommand{\cC}{\mathcal{C}}\newcommand{\cD}{\mathcal{D}}\newcommand{\cE}{\mathcal{E}}\newcommand{\cF}{\mathcal{F}}\newcommand{\cG}{\mathcal{G}}\newcommand{\cH}{\mathcal{H}}\newcommand{\cI}{\mathcal{I}}\newcommand{\cJ}{\mathcal{J}}\newcommand{\cK}{\mathcal{K}}\newcommand{\cL}{\mathcal{L}}\newcommand{\cM}{\mathcal{M}}\newcommand{\cN}{\mathcal{N}}\newcommand{\cO}{\mathcal{O}}\newcommand{\cP}{\mathcal{P}}\newcommand{\cQ}{\mathcal{Q}}\newcommand{\cR}{\mathcal{R}}\newcommand{\cS}{\mathcal{S}}\newcommand{\cT}{\mathcal{T}}\newcommand{\cU}{\mathcal{U}}\newcommand{\cV}{\mathcal{V}}\newcommand{\cW}{\mathcal{W}}\newcommand{\cX}{\mathcal{X}}\newcommand{\cY}{\mathcal{Y}}\newcommand{\cZ}{\mathcal{Z}}

\newcommand{\sA}{\mathscr{A}}\newcommand{\sB}{\mathscr{B}}\newcommand{\sC}{\mathscr{C}}\newcommand{\sD}{\mathscr{D}}\newcommand{\sE}{\mathscr{E}}\newcommand{\sF}{\mathscr{F}}\newcommand{\sG}{\mathscr{G}}\newcommand{\sH}{\mathscr{H}}\newcommand{\sI}{\mathscr{I}}\newcommand{\sJ}{\mathscr{J}}\newcommand{\sK}{\mathscr{K}}\newcommand{\sL}{\mathscr{L}}\newcommand{\sM}{\mathscr{M}}\newcommand{\sN}{\mathscr{N}}\newcommand{\sO}{\mathscr{O}}\newcommand{\sP}{\mathscr{P}}\newcommand{\sQ}{\mathscr{Q}}\newcommand{\sR}{\mathscr{R}}\newcommand{\sS}{\mathscr{S}}\newcommand{\sT}{\mathscr{T}}\newcommand{\sU}{\mathscr{U}}\newcommand{\sV}{\mathscr{V}}\newcommand{\sW}{\mathscr{W}}\newcommand{\sX}{\mathscr{X}}\newcommand{\sY}{\mathscr{Y}}\newcommand{\sZ}{\mathcal{Z}}

\newcommand{\al}{\alpha}\newcommand{\be}{\beta}\newcommand{\ga}{\gamma}\newcommand{\de}{\delta}\newcommand{\ve}{\varepsilon}\newcommand{\et}{\eta}\newcommand{\ph}{\phi}\newcommand{\vp}{\varphi}\newcommand{\ps}{\psi}\newcommand{\ka}{\kappa}\newcommand{\la}{\lambda}\newcommand{\om}{\omega}\newcommand{\rh}{\rho}\newcommand{\si}{\sigma}\newcommand{\tht}{\theta}\newcommand{\ta}{\tau}\newcommand{\ch}{\chi}\newcommand{\ze}{\zeta}\newcommand{\Ga}{\Gamma}\newcommand{\De}{\Delta}\newcommand{\Ph}{\Phi}\newcommand{\Ps}{\Psi}\newcommand{\La}{\Lambda}\newcommand{\Om}{\Omega}\newcommand{\Si}{\Sigma}\newcommand{\Tht}{\Theta}\newcommand{\na}{\nabla}

\newcommand{\nm}[1]{\left\|#1\right\|}\newcommand{\md}[1]{\left|#1\right|}\newcommand{\Md}[1]{\Big|#1\Big|}\newcommand{\db}[1]{[\![#1]\!]}
\newcommand{\vol}{\operatorname{Vol}}\newcommand{\tp}{\tilde{\Ph}}\newcommand{\tps}{\tilde{\Ps}}\newcommand{\tx}{\tilde{X}}\newcommand{\ty}{\tilde{Y}}\newcommand{\zt}{\tilde{Z}}

\vspace{-2ex}
\begin{abstract}
	\vspace{-1.5ex}
	\noindent
Let \((M,g)\) be a closed Riemannian manifold and \(\{\om_p\}_{p=1}^{\infty}\) be the volume spectrum of \((M,g)\). We will show that \(\om_{k+m+1}\leq \om_k+\om_m+W\) for all \(k,m\geq 0\), where \(\om_0=0\) and \(W\) is the one-parameter Almgren-Pitts width of \((M,g)\). We will also prove the similar inequality for the $\ve$-phase-transition spectrum \(\{c_{\ve}(p)\}_{p=1}^{\infty}\) using the Allen-Cahn approach.
\end{abstract}

\section{Introduction}
For a closed \Rm manifold \((M^{n+1},g)\), the spectrum of the Laplacian on \((M,g)\), denoted by \(\{\la_p\}_{p=1}^{\infty}\), \(0<\la_1\leq \la_2\leq\dots\), has the following min-max characterization. Let \(H^1(M)\) be the Sobolev space of real valued functions \(\vp \in L^2(M)\) \st \(\na \vp\), the distributional derivative of \(\vp\), is also in \(L^2(M)\). \sps \(\cV_p\) denotes the set of all \(p\)-dimensional vector subspaces of \(H^1(M)\). Then
\begin{equation}\label{e.lampda_p}
\la_p=\inf_{V\in \cV_p}\sup_{f\in V\setminus\{0\}} R(f)\quad \text{ where }\quad R(f)=\frac{\int_{M}\md{\na f}^2}{\int_{M}f^2}.
\end{equation}
The Rayleigh quotient \(R(f)\) is invariant under the scaling, i.e. for all \(c\in \bbr\setminus\{0\}\), \(R(f)=R(cf)\). \hn it descends to a well-defined functional on \(\mathbb{P}(H^1(M))\), the projective space associated to the real vector space \(H^1(M)\).

In \cite{gro,gro1,gro2}, Gromov introduced various non-linear analogues of the spectrum of the Laplacian. In the context of the area functional and the minimal hypersurfaces, the relevant spectrum is the volume spectrum \(\{\om_p\}_{p=1}^{\infty}\). \(\om_p\) is defined by a min-max quantity, similar to \eqref{e.lampda_p}. One replaces the vector space \(H^1(M)\) by \(\zn\), the space of mod \(2\) flat hypercycles which bound a region in \(M\) and the Rayleigh quotient by the area functional (see Section \ref{s2.2} for the precise definition). By the works of Almgren \cite{alm_article} and Marques-Neves \cite{mn_morse}, the space \(\zn\) is weakly homotopy equivalent to \(\bbr\mathbb{P}^{\infty}\). The cohomology ring \(H^*(\zn,\bbz_2)\) is the polynomial ring \(\bbz_2[\overline{\la}]\) where \(\overline{\la}\in H^1(\zn,\bbz_2) \). In the definition of \(\om_p\), instead of considering \(p\)-dimensional vector subspaces of \(H^1(M)\), one considers all the \(\sS\subset\zn\) \st \(\overline{\la}^p\big|_{\sS}\neq 0\) in \(H^p(\sS,\bbz_2).\)

The connection between the volume spectrum and the minimal hypersurfaces comes from the Almgren-Pitts min-max theory, developed by Almgren \cite{alm}, Pitts \cite{pitts}, Schoen-Simon \cite{ss}. If \(\Pi\) is a homotopy class of maps \(X \ra \zn\), by the Almgren-Pitts min-max theory, the width of \(\Pi\) is achieved by the area of a closed, minimal hypersurface (which can have a singular set of Hausdorff dimension \(\leq n-7\) if the ambient dimension \(n+1\geq 8\)), possibly with multiplicities. The index upper bound of the min-max minimal hypersurfaces was proved by Marques-Neves \cite{MN_index} and Li \cite{Li_index}. Sharp \cite{Sharp} proved the compactness of the space of minimal hypersurfaces with bounded area and index, which holds in higher dimensions as well \cite{d0}. Combining all these, one can conclude that \cite{imn}*{Proposition 2.2} for each \(p \in \bbn\), \(\om_p\) is achieved by the area of a closed, minimal hypersurface with optimal regularity (possibly with multiplicities), whose index is bounded above by \(p\); see also \cite{li}*{Corollary 3.2} where Li gave a different argument. Moreover, if the ambient dimension \(3\leq n+1\leq 7\), by the works of Marques-Neves \cite{mn_morse} and Zhou \cite{zhou2}, for a generic (bumpy) metric, \(\om_p\) is realized by the area of a closed, two-sided, minimal hypersurface with multiplicity one, whose index \(=p\). When the ambient dimension is \(3\), this was also proved by Chodosh and Mantoulidis \cite{cm}, in the Allen-Cahn setting.

By the works of Gromov \cite{gro,gro1}, Guth \cite{guth}, Marques and Neves \cite{mn_ricci_positive}, there exist positive constants \(C_1\), \(C_2\), depending on the \mt \(g\), \st
\begin{equation}\label{e.sublin.growth}
C_1p^{\frac{1}{n+1}}\leq \om_p\leq C_2p^{\frac{1}{n+1}}\quad \forall p\in \bbn.
\end{equation}
In \cite{lmn}, Liokumovich, Marques and Neves proved that \(\{\om_p\}_{p=1}^{\infty}\) satisfies the following Weyl law, which was conjectured by Gromov \cite{gro1}.
\begin{equation}\label{e.weyl.law}
\lim_{p\ra\infty}\om_p(M,g)p^{-\frac{1}{n+1}}=a(n)\vol(M,g)^{\frac{n}{n+1}},
\end{equation}
where \(a(n)>0\) is a constant, which depends only on the ambient dimension.

The above mentioned theorems regarding the asymptotic behaviour of \(\{\om_p\}_{p=1}^{\infty}\), has a number of important applications, which we discuss now. By the works of Marques-Neves \cite{mn_ricci_positive} and Song \cite{song}, every closed \mf \((M^{n+1},g)\), \(3\leq n+1\leq 7\), contains infinitely many closed minimal hypersurfaces, which confirms a conjecture of Yau \cite{yau}. In \cite{imn}, Irie, Marques and Neves proved that for a generic metric \(g\) on \(M^{n+1}\), \(3\leq n+1\leq 7\), the union of all closed minimal hypersurfaces is dense in \(M\). This theorem was quantified in \cite{mns} by Marques, Neves and Song, where they proved that for a generic \mt \(g\), there exists an equidistributed sequence of closed minimal hypersurfaces in \((M,g)\). Recently, Song and Zhou \cite{SZ} proved the generic scarring phenomena for minimal hypersurfaces, which can be interpreted as the opposite of the equidistribution phenomena. (The arguments in the papers \cite{song} and \cite{SZ} use the Weyl law for the volume spectrum on certain \textit{non-compact manifolds} with cylindrical ends.) In higher dimensions, Li \cite{li} proved the existence of infinitely many closed minimal hypersurfaces (with optimal regularity) for a generic set of metrics.

The properties of the volume spectrum also turn out to be useful in the context of constant mean curvature (CMC) hypersurfaces. In \cite{zz1}, Zhou and Zhu developed the min-max theory for CMC hypersurfaces, which was further extended by Zhou \cite{zhou2}. In particular, they \cite{zz1} proved that for all \(c>0\), every closed manifold \((M^{n+1},g)\), \(n+1\geq 3\), contains a closed \(c\)-CMC hypersurface (with optimal regularity). Building on the works in \cite{zz1} and \cite{zhou2}, we proved in \cite{d1} that the number of closed \(c\)-CMC hypersurfaces (with optimal regularity) in \((M^{n+1},g)\) is at least \(\varrho_0c^{-\frac{1}{n+1}}\), where \(\varrho_0>0\) is a constant, which depends on the \mt \(g\). To obtain this estimate, we used the lower bound of \(\om_p\), stated in \eqref{e.sublin.growth} and the following inequality.
\begin{equation}\label{e.easy.subadd}
\om_{p+1}\leq \om_p +W \quad \forall p\in \bbn,
\end{equation}
where \(W\) is the one parameter Almgren-Pitts width of \((M,g)\).

In the present article, we will prove a more general sub-additive inequality for the volume spectrum, as stated below.
\begin{thm}\label{thm 1.1}
Let \((M,g)\) be a closed Riemannian manifold. Let \(\{\om_p\}_{p=1}^{\infty}\) be the volume spectrum and \(W\) be the one parameter Almgren-Pitts width of \((M,g)\) (see Section \ref{s2.2}). Then, for all \(k,m\geq 0\),
\begin{equation}\label{eq1.1}
\om_{k+m+1}\leq \om_k+\om_m+W,
\end{equation}
where we set \(\om_0=0.\)
\end{thm}
A similar inequality also holds for the phase-transition spectrum (as stated below in Theorem \ref{thm 1.2}), which is the Allen-Cahn analogue of the volume spectrum. The Allen-Cahn min-max theory is a PDE based approach to the min-max construction of minimal hypersurfaces, which was introduced by Guaraco \cite{Guaraco} and further extended by Gaspar and Guaraco \cite{GG1}. The regularity of the minimal hypersurfaces, obtained from the Allen-Cahn theory, depends on the previous works by Hutchinson-Tonegawa \cite{HT}, Tonegawa \cite{t}, Wickramasekera \cite{w} and Tonegawa-Wickramasekera \cite{TW}. In \cite{GG1}, Gaspar and Guaraco defined the phase-transition spectrum and proved a sub-linear growth estimate for them (which is similar to \eqref{e.sublin.growth}). In \cite{GG2}, they proved a Weyl law for the phase-transition spectrum (which is similar to the Weyl law for the volume spectrum \eqref{e.weyl.law}).
\begin{thm}\label{thm 1.2}
	Let \((M,g)\) be a closed Riemannian manifold. For \(\ve>0\), let \( \{c_{\ve}(p)\}_{p=1}^{\infty} \) be the \(\ve\)-phase-transition spectrum and \(\ga_{\ve}\) be the one parameter \(\ve\)-Allen-Cahn width of \((M,g)\) (see Section \ref{s2.3}). Then, for all \(k,m\geq 0\),
	\begin{equation}\label{eq1.2}
	c_{\ve}(k+m+1)\leq c_{\ve}(k)+c_{\ve}(m)+\ga_{\ve},
	\end{equation}	
	where \(c_{\ve}(0)=0.\) Hence, letting \(\ve\ra 0^+\), we obtain the following inequality for the phase-transition spectrum \(\{\ell_p\}_{p=1}^{\infty}\).
\[\ell_{k+m+1}\leq \ell_k+\ell_m+\ga,\]
where \(\ell_0=0\) and \(\ga\) is the one parameter Allen-Cahn width of \((M,g)\).
\end{thm}

\textbf{Acknowledgements.}  I am very grateful to my advisor Prof. Fernando Cod\'{a} Marques for many helpful discussions and for his constant support and guidance. The author is partially supported by NSF grant DMS-1811840.

\section{Notation and Preliminaries}
\subsection{Caccioppoli sets}
In this subsection, we will briefly recall the notion of the Caccioppoli set; further details can be found in \cite{sim} and \cite{afp}. An \(\cH^{n+1}\)-measurable set \(E \subset (M,g)\) is called a {\it Caccioppoli set} if \(D\ch_E\), the distributional derivative of the characteristic function \(\ch_E\), is a finite Radon measure on \(M\). This is equivalent to
\[\sup \left\{\int_E \operatorname{div} \om \; d\cH^{n+1}:\om \in \mathfrak{X}^1(M),\; \|\om\|_{\infty}\leq 1\right\}<\infty;\]
where \(\mathfrak{X}^1(M)\) denotes the space of \(C^1\) vector-fields on \(M\). Let us use the notation \(\cm\) to denote the space of all Caccioppoli sets in \(M\). If \(E\in \cm\), there exists an \(n\)-rectifiable set \(\del E\) \st the total variation measure \(|D\ch_E|=\cH^n \mres \del E\). Hence, for all \(\om \in \mathfrak{X}^1(M),\)
\[\int_E \operatorname{div} \om \; d\cH^{n+1}=\int_{\del E}\langle \om,\nu_E\rangle\;d\cH^n,\]
where \(\nu_E\) is a \(|D\ch_E|\)-measurable vector-field; \(\|\nu_E\|=1\) \(|D\ch_E|\)-a.e. The following proposition can be found in \cite{afp}. For the sake of completeness, we also include its proof (following \cite{afp}).
\begin{pro}[\cite{afp}*{Proposition 3.38}]\label{pr cac set}
Let \(E \subset M\) be \(\cH^{n+1}\)-measurable and \(U \subset M\) be an open set. Let
\[P(E,U)=\sup \left\{\int_E \operatorname{div} \om \; d\cH^{n+1}:\om \in \mathfrak{X}^1_c(U),\; \|\om\|_{\infty}\leq 1\right\};\]
where \(\mathfrak{X}^1_c(U)\) denotes the space of compactly supported \(C^1\) vector-fields on \(U\). If \(E,F \in \cm\),
\begin{equation}\label{e cac set}
P(E\cap F,U)+P(E\cup F, U)\leq P(E,U)+P(F,U).
\end{equation}
\hn if \(E,F \in \cm\), \(E\cap F\) and \(E \cup F\) also belong to \(\cm.\)
\end{pro}
\begin{proof}
\Sn \(E,F\in \cm\), by \cite{MPPP}*{Proposition 1.4}, \te \(\{f_i\}_{i=1}^{\infty},\{g_i\}_{i=1}^{\infty} \subset C^{\infty}(U)\) \st \(0\leq f_i,g_i\leq 1\) for all \(i\);
\begin{align*}
& f_i \ra \ch_E\big|_U \text{ in } L^1(U) \text{ and pointwise a.e.}; \quad P(E,U)=\lim_{i \ra \infty}\int_U |\na f_i|\; d\cH^{n+1};\\
& g_i \ra \ch_F\big|_U \text{ in } L^1(U) \text{ and pointwise a.e.}; \quad P(F,U)=\lim_{i \ra \infty}\int_U |\na g_i|\; d\cH^{n+1}.
\end{align*}
By the dominated convergence theorem, 
\[f_ig_i \ra \ch_{E \cap F}\big|_U \text{ in }L^1(U) \;\text{  and  }\; f_i+g_i-f_ig_i\ra \ch_{E \cup F}\big|_U \text{ in }L^1(U).\]
\tf
\begin{align*}
&P(E\cap F, U)+P(E\cup F, U)\\ &\leq \liminf_{i \ra \infty}\int_{U}\Big(\big|{\na(f_ig_i)}\big|+\big|{\na f_i +\na g_i - \na (f_ig_i)}\big|\Big)\; d\cH^{n+1}\\
&\leq \liminf_{i \ra \infty}\int_{U}\Big(g_i\md{\na f_i}+f_i\md{\na g_i}+(1-g_i)\md{\na f_i}+(1-f_i)\md{\na g_i}\Big)\; d\cH^{n+1}\\
&=P(E,U)+P(F,U).
\end{align*}
\end{proof}

\subsection{The space of hypercycles and the volume spectrum}\label{s2.2}
For \(l \in \bbn\), let \(\mathbf{I}_l(M^{n+1};\bbz_2)\) be the space of \(l\)-dimensional flat chains in \(M\) with coefficients in \(\bbz_2\). We will only need to consider \(l=n,n+1\). \(\cZ_n(M^{n+1};\bbz_2)\) denotes the space of flat chains \(T \in \mathbf{I}_n(M;\bbz_2)\) \st \(T =\del \Om\) for some \(\Om \in \mathbf{I}_{n+1}(M;\bbz_2)\). If \(T \in \mathbf{I}_n(M;\bbz_2)\), \(\md{T}\) stands for the varifold associated to \(T\) and \(\nm{T}\) is the Radon measure associated to \(\md{T}\). \(\cF\) and \(\bM\) denote the flat norm and the mass norm on \(\mathbf{I}_l(M;\bbz_2)\). When \(l=n+1\), these two norms coincide. We will always assume that the spaces \(\mathbf{I}_{n+1}(M;\bbz_2)\) and \(\zn\) are equipped with the \(\cF\) norm. We will also identify \(\cm\) with \(\mathbf{I}_{n+1}(M;\bbz_2)\), i.e. \(E \in \cm\) will be identified with \(\db{E}\), the current associated with \(E\). Similarly, \(\del E\) (with \(E \in \cm\)) will be identified with \(\db{\del E}=\del \db{E}\).

In \cite{mn_morse}, Marques and Neves proved that the space \(\cm\) is contractible and the boundary map \(\del:\cm\ra \zn\) is a double cover. Indeed, by the constancy theorem, for \(\Om_1,\Om_2\in \cm\), \(\del\Om_1=\del\Om_2\) if and only if either \(\Om_1=\Om_2\) or \(\Om_1=M-\Om_2.\) Thus \(\pi_1(\zn)=\bbz_2\). It was also proved in \cite{mn_morse} that \(\zn\) is weakly homotopy equivalent to \(\bbr\mathbb{P}^{\infty}\). The cohomology ring \(H^*(\zn,\bbz_2)\) is the polynomial ring \(\bbz_2[\overline{\la}]\) where \(\overline{\la}\) is the unique non-zero cohomology class in \(H^1(\zn,\bbz_2)\).

\(X\) is called a  cubical complex if \(X\) is a subcomplex of \([0,1]^N\) for some \(N\in \bbn\). By \cite{bp}*{Chapter 4}, every cubical complex is homeomorphic to a finite simplicial complex and vice-versa. In the present article, we choose to work with simplicial complexes.

Let \(X\) be a finite simplicial complex. \sps \(\Ph:X \ra \zn\) is a continuous map. \(\Ph\) is called a \textit{\(p\)-sweepout} if \(\Ph^*(\overline{\la}^p)\neq 0\) in \(H^p(X,\bbz_2)\). \(\Ph\) is said to have {\it no concentration of mass} if 
\[\lim_{r \ra 0^+}\sup\left\{\nm{\Ph(x)}(B(p,r)):x\in X,\;p\in M\right\}=0.\]
Let \(\cP_p\) denote the set of all \(p\)-sweepouts with no concentration of mass. The {\it volume spectrum} \(\{\om_p\}_{p=1}^{\infty}\) is defined by
\[\om_p=\inf_{\Ph\in\cP_p}\sup \left\{\bM(\Ph(x)):x\in \text{domain of }\Phi\right\}.\]

Let us also recall the definition of the one parameter Almgren-Pitts width, denoted by \(W\). Suppose \(\cS\) is the set of all \cts maps \(\La:[0,1]\ra \cm\) \st \(\La(0)=M\), \(\La(1)=\emptyset\) and \(\del\circ\La:[0,1]\ra\zn\) has no concentration of mass. We define
\[W=\inf_{\La\in\cS}\sup\left\{\bM(\del\La(t)):t \in [0,1]\right\}.\]

\subsection{The phase-transition spectrum}\label{s2.3}
Here we briefly recall the definition of the phase-transition spectrum, which was originally defined by Gaspar and Guaraco in \cite{GG1}. Let \(\mathfrak{p}:E\ra B\) be a universal \(\bbz_2\)-principal bundle. Then \(B\) is homotopy equivalent to \(\bbr\mathbb{P}^{\infty}\); hence \(H^*(B,\bbz_2)\) is isomorphic to the polynomial ring \(\bbz_2[\xi]\), where \(\xi\in H^1(B,\bbz_2)\). Suppose \(\tx\) is a finite simplicial complex with a free, simplicial \(\bbz_2\) action (i.e. \(\bbz_2\) acts on $\tx$ by simplicial homeomorphism) and \(X\) is the quotient of \(\tx\) by \(\bbz_2\). There exists a \cts \(\bbz_2\)-equivariant map \(\tilde{f}:\tx \ra E\), which is unique up to \(\bbz_2\)-homotopy (see \cite{Dieck}*{Chapter 14.4}). \(\tilde{f}\) descends to a map \(f:X \ra B\), i.e. if \(\pi:\tx \ra X\) is the projection map, \(f\circ \pi=\mathfrak{p}\circ \tilde{f}\). Let \(f^*:H^*(B, \bbz_2)\ra H^*(X, \bbz_2)\) be the map induced by \(f\). One defines
\[\text{Ind}_{\bbz_2}(\tx)=\sup\left\{p\in \bbn:f^*(\xi^{p-1})\neq 0 \in H^{p-1}(X, \bbz_2)\right\}.\]
\(\cC_p\) denotes the set of all finite simplicial complex \(\tx\), with free, simplicial \(\bbz_2\) action, \st \(\text{Ind}_{\bbz_2}(\tx)\geq p+1.\) 

Let us recall that 
\[H^1(M)=\left\{\vp \in L^2(M):\text{the distributional derivative }\na \vp\in L^2(M)\right\}.\]
\(H^1(M)\) is an infinite dimensional, separable Hilbert space; hence \(H^1(M)\setminus \{0\}\) is contractible. There is a free \(\bbz_2\) action on \(H^1(M)\setminus \{0\}\) given by \(u \mapsto -u\). Thus \(H^1(M)\setminus \{0\}\), equipped with this \(\bbz_2\) action, is the total space of a universal \(\bbz_2\)-principal bundle (see \cite{Dieck}*{14.4.12}). If \(\tx \) is a finite simplicial complex with free, simplicial \(\bbz_2\) action, the set of all \cts \(\bbz_2\)-equivariant maps \(\tx \ra H^1(M)\setminus \{0\}\) is denoted by \(\Ga(\tx)\).

For $\ve>0$, the \textit{\(\ve\)-Allen-Cahn functional} \(E_{\ve}:H^1(M)\ra [0,\infty)\) is defined by
\[E_{\ve}(u)=\int_M\ve\frac{\md{\na u}^2}{2}+\frac{W_*(u)}{\ve},\]
where \(W_*:\bbr \ra \bbr\) is a smooth, symmetric double-well potential. More precisely, \(W_*\) is bounded; \(W_*(t)=W_*(-t)\) for all \(t\); \(W_* \geq 0\) and \(W_*\) has precisely three critical points \(0,\pm 1\); \(W_*(\pm 1)=0\) and \(W_*''(\pm 1)>0\); \(0\) is a local maximum of \(W_*\). We note that \(E_{\ve}(u)=E_{\ve}(-u)\). The \textit{\(\ve\)-phase-transition spectrum} \(\{c_{\ve}(p)\}_{p=1}^{\infty}\) can be defined as follows (see \cite{GG1}*{Lemma 6.2}).
\[c_{\ve}(p)=\inf_{\tx\in \cC_p}\left(\inf_{h\in \Ga(\tx)}\sup_{x\in\tx}E_{\ve}(h(x))\right).\]
The {\it phase-transition spectrum} \(\{\ell_p\}_{p=1}^{\infty}\) is defined by 
\[\ell_p=\frac{1}{2\si}\lim_{\ve \ra 0^+}c_{\ve}(p),\]
where \(\si=\int_{-1}^{1}\sqrt{W(t)/2}\;dt.\) Combining the results in \cite{Guaraco}, \cite{GG1} and \cite{d2}, it follows that \(\ell_p=\om_p\) for all \(p\). Let us also recall the definition of the \textit{one parameter Allen-Cahn width \(\ga\)}, which was defined by Guaraco \cite{Guaraco}. Suppose
\[\Ga=\left\{h:[0,1]\ra H^1(M):h \text{ is continuous, }h(0)=\mathbf{1},\; h(1)=-\mathbf{1}\right\}\]
(\(\mathbf{1}\) denotes the constant function \(1\)). Then
\[\ga_{\ve}=\inf_{h\in \Ga}\sup_{t\in[0,1]}E_{\ve}(h(t));\quad \ga=\frac{1}{2\si}\lim_{\ve \ra 0^+}\ga_{\ve}.\]
Again, by the results in \cite{Guaraco} and \cite{d2}, \(\ga=W\).

\subsection{Join of two topological spaces}\label{s2.4}
Let \(A\) and \(B\) be two topological spaces. \textit{The join of \(A\) and \(B\)}, denoted by \(A*B\), is defined as follows.
\[A*B=\frac{A\times B \times [0,1]}{\sim},\]
where the equivalence relation `\(\sim\)' is such that
\[(a,b_1,0)\sim (a, b_2, 0)\;\; \forall\;b_1,b_2 \in B \quad \text{ and }\quad (a_1, b, 1) \sim (a_2, b,1)\;\; \forall \; a_1, a_2 \in A,\]
i.e. $A\times B \times \{0\}$ is collapsed to $A$ and $A\times B \times \{1\}$ is collapsed to $B$. If either \(A\) or \(B\) is path-connected, using van Kampen’s Theorem, one can show that \(A*B\) is simply connected.

By the abuse of notation, an element of \(A*B\) will be denoted by a triple \((a,b,t)\) with \(a\in A,\;b \in B,\; t\in [0,1]\). If \(f:A \ra C\) and \(g:B \ra D\) are \cts maps, one can define a \cts map \(f*g:A*B \ra C*D\) as follows.
\[(f*g)(a,b,t)=\left(f(a),g(b),t\right); \quad a \in A,\;b\in B,\; t \in [0,1].\]

Let us recall that the standard \(d\)-simplex \(\De^d\) is defined as follows.
\[\De^d=\{x=(x_1,x_2,\dots,x_{d+1})\in \bbr^{d+1}:x_1+x_2+\dots+x_{d+1}=1\;\text{ and }x_i\geq 0 \;\;\forall i\}.\]
\(\De^p*\De^q\) can be identified with the simplex \(\De^{p+q+1}\). Indeed, \(F:\De^p*\De^q\ra \De^{p+q+1}\), defined by
\[F(u,v,t)=((1-t)u,tv)\;(\in \bbr^{p+q+2}),\]
is a homeomorphism. \hn join of two simplicial complexes is again a simplicial complex. 

If \(\mathcal{X}\) is a topological space, let \(C_d(\mathcal{X},\bbz_2)\) denote the abelian group of singular \(d\)-chains in \(\mathcal{X}\) with coefficients in \(\bbz_2.\) For \(p,q\in \bbn_0\), we define a bilinear map \(\sF:C_p(A, \bbz_2)\times C_q(B, \bbz_2)\ra C_{p+q+1}(A*B, \bbz_2)\) as follows. If \(\ph:\De^p \ra A\) and \(\ps:\De^q \ra B\) are \cts maps,
\[\sF(\ph,\ps)=(\ph*\ps)\circ \left(F^{-1}\right).\]
For convenience, if \((\ph,\ps) \in C_p(A, \bbz_2)\times C_q(B, \bbz_2) \), let us denote \(\sF(\ph,\ps)\) simply by \(\ph*\ps\). The following identity can be verified by explicit computation.
\begin{equation}\label{e join of sing simpl}
\del(\ph*\ps)=(\del\ph)*\ps+\ph*(\del\ps).
\end{equation}

\section{Proof of Theorem \ref{thm 1.1}}
If \(k=m=0\), \eqref{eq1.1} reduces to \(\om_1\leq W\), which holds by the definitions of \(\om_1\) and \(W\). So let us assume that \(k\geq 1.\) The proof of Theorem \ref{thm 1.1} is divided into four parts.

\textbf{Part 1.} Let us fix \(\de>0.\) We choose a \(k\)-sweepout \(\Ph_1:X \ra \zn\), with no concentration of mass, \st \(X\) is connected and
\begin{equation}\label{e mass of Phi_1}
\sup_{x\in X}\{\bM(\Ph_1(x))\}\leq \om_k+\de.
\end{equation}
Following the argument of Zhou \cite{zhou2}, \(\Ph_1\) is a \(k\)-sweepout implies that 
\[\Ph_1^*:H^1(\zn,\bbz_2)\ra H^1(X, \bbz_2)\]
is non-zero. Hence,
\[(\Ph_1)_*:\pi_1(X)\ra \pi_1(\zn)(=\bbz_2)\]
is onto. Thus \(\ker (\Ph_1)_*\) is an index \(2\) subgroup of \(\pi_1(X)\). By \cite{hat}*{Proposition 1.36}, there exists a double cover \(\rho_1:\tilde{X} \ra X\) \st \(\tilde{X}\) is connected and if 
\[(\rh_1)_*:\pi_1(\tx)\ra \pi_1(X),\]
then \(\operatorname{im}(\rh_1)_*=\ker (\Ph_1)_*\). Using \cite{hat}*{Proposition 1.33}, \(\Ph_1\) has a lift \(\tilde{\Ph}_1:\tilde{X}\ra \cm\) \st \(\del \circ \tilde{\Ph}_1=\Ph_1 \circ \rho_1.\) \(\tilde{\Ph}_1\) is \(\bbz_2\)-equivariant, i.e. if \(T_1:\tilde{X} \ra \tilde{X}\) is the deck transformation,
\[\tilde{\Ph}_1(T_1(x))=M-\tilde{\Ph}_1(x) \;\; \forall\; x\in \tilde{X}.\]
Similarly, for \(m\geq 1\), we can choose an \(m\)-sweepout \(\Ph_2:Y \ra \zn\) with no concentration of mass \st \(Y\) is connected and
\begin{equation}\label{e mass of Phi_2}
\sup_{y\in Y}\{\bM(\Ph_2(y))\}\leq \om_m+\de.
\end{equation}
There exists a double cover \(\rho_2:\tilde{Y} \ra Y\) \st \(\tilde{Y}\) is connected and \(\Ph_2\) has a lift \(\tilde{\Ph}_2:\tilde{Y}\ra \cm\) \st \(\del \circ \tilde{\Ph}_2=\Ph_2 \circ \rho_2.\) If \(T_2:\tilde{Y} \ra \tilde{Y}\) is the deck transformation, then
\[\tilde{\Ph}_2(T_2(y))=M-\tilde{\Ph}_2(y) \;\; \forall\; y\in \tilde{Y}.\]
On the other hand, if \(m=0\), we simply choose \(Y\) to be a singleton set \(\{\bullet\}\) and \(\Ph_2(\bullet)\) to be the zero cycle. Consequently, in this case, \(\ty=\{-1,1\}\) with \(\tp_2(-1)=\emptyset\) and \(\tp_1(1)=M\). We also choose \(\La:[0,1] \ra \cm \) \st \(\La(0)=M\), \(\La(1)=\emptyset\), \(\del \circ \La:[0,1]\ra \zn\) has no concentration of mass and
\begin{equation}\label{e mass of del Laambda}
\sup_{t \in [0,1]}\{\bM(\del\La(t))\}\leq W+\de.
\end{equation}

{\bf Part 2.} Let us consider \(\tx*\ty\). As remarked in Section \ref{s2.4}, \(\tx*\ty\) is simply-connected. There exists a free \(\bbz_2\) action on \(\tx*\ty\) given by 
\begin{equation}\label{e.Z_2.action}
(x,y,t)\mapsto(T_1(x),T_2(y),t).
\end{equation}
\(\tx\) and \(\ty\) can be equipped with simplicial complex structures \st the maps \(T_1:\tx\ra \tx\) and \(T_2:\ty\ra \ty\) are simplicial homeomorphisms. Let \(\{\De_{1,i}:i=1,\dots,I\}\) and \(\{\De_{2,j}:j=1,\dots,J\}\) be the set of simplices in \(\tx\) and \(\ty\). As discussed in Section \ref{s2.4}, we can put a canonical simplicial complex structure on \(\tx*\ty\), whose simplices are \(\{\De_{1,i},\;\De_{2,j},\;\De_{1,i}*\De_{2,j}:i=1,\dots,I;\;j=1,\dots,J\}\), \st the above action of \(\bbz_2\) on \(\tx*\ty\) (\eqref{e.Z_2.action}) is a simplicial action.

One can define a \(\bbz_2\)-equivariant map \(\tilde{\Ps}:\tx*\ty\ra \cm\) as follows.
\begin{equation}\label{e def of Psi tilde}
\tilde{\Ps}(x,y,t)=\Big(\tp_1(x) \cap \La(t)\Big) \bigcup \Big(\tp_2(y) \cap (M-\La(t))\Big).
\end{equation}
This definition is motivated by the work of Marques and Neves \cite{mn_morse}*{Proof of Claim 5.3}, where they proved that \(\cm\) is contractible. By Proposition \ref{pr cac set}, R.H.S. of \eqref{e def of Psi tilde} indeed belongs to \(\cm\). Moreover, \(\tps\) is well-defined on \(\tx*\ty\) as \(\La(0)=M\) and \(\La(1)=\emptyset\).
\begin{clm}\label{cl ctn in flat}
The map \(\tilde{\Ps}:\tx*\ty\ra \cm\), defined above, is \cts in the flat topology.
\end{clm}
\begin{proof}
\Sn \(\tp_1\), \(\tp_2\) and \(\La\) are \cts in the flat topology, it is enough to prove that the maps \(\mu_1,\mu_2:\cm \times \cm \ra \cm\) defined by
\[\mu_1(E,F)=E\cap F\quad \text{  and  }\quad \mu_2(E,F)=E \cup F\]
are continuous. For arbitrary sets \(A_1,\;A_2,\;B_1,\;B_2\),
\[(A_1 \cap A_2) \; \De \; (B_1 \cap B_2)\subset (A_1\;\De\;B_1)\cup (A_2\;\De\;B_2),\]
and
\[(A_1 \cup A_2) \; \De \; (B_1 \cup B_2)\subset (A_1\;\De\;B_1)\cup (A_2\;\De\;B_2).\]
Here \(A \; \De\;B=(A\setminus B)\cup (B \setminus A)\) is the symmetric difference of \(A\) and \(B\). \tf
\begin{align*}
&\cF(\mu_1(E,F)-\mu_1(E',F'))\leq \cF(E-E')+\cF(F-F'),\\
&\cF(\mu_2(E,F)-\mu_2(E',F'))\leq \cF(E-E')+\cF(F-F'),
\end{align*}
which proves the continuity of \(\mu_1\) and \(\mu_2.\)
\end{proof}
Let \(Z\) denote the quotient of \(\tx*\ty\) under the above mentioned \(\bbz_2\) action \eqref{e.Z_2.action} (\(Z\) is a simplicial complex as the \(\bbz_2\) action on \(\tx*\ty\) is simplicial) and \(\rho:\tx*\ty \ra Z\) be the covering map. \Sn \(\tilde{\Ps}:\tx*\ty\ra \cm\) is \(\bbz_2\)-equivariant, \(\tilde{\Ps}\) descends to a \cts map \(\Ps:Z \ra \zn\), i.e. \(\del\circ \tilde{\Ps}=\Ps\circ\rho.\) If \(E \in \cm\), \(P(E,U)=\|\del E\|(U)\). \Sn \(\Ph_1,\;\Ph_2,\;\del\circ\La\) have no concentration of mass, from Proposition \ref{pr cac set}, it follows that \(\Ps\) also has no concentration of mass.

{\bf Part 3.} We will show that \(\Ps:Z \ra \zn\) is a \((k+m+1)\)-sweepout. \Sn \(\tx*\ty\) is simply-connected and \(Z\) is the quotient of \(\tx*\ty\) by \(\bbz_2\), \(\pi_1(Z)=\bbz_2\); hence \(H^1(Z,\bbz_2)=\bbz_2.\) Let \(\la\) be the unique non-zero element of \(H^1(Z,\bbz_2)\) so that
\begin{equation}\label{e labmda.c}
\la.[\ka]=1
\end{equation}
for all non-contractible loop \(\ka:S^1\ra Z.\)

As discussed in Section \ref{s2.4}, \(\tx\) naturally sits inside \(\tx*\ty\), via the map \(x \mapsto (x,\overline{y},0)\) and \(\ty\) naturally sits inside \(\tx*\ty\), via the map \(y \mapsto (\overline{x},y,1)\). \hn \(X\) and \(Y\) can be naturally identified as subspaces of \(Z\); let \(\iota_1:X \hookrightarrow Z\) and \(\iota_2:Y \hookrightarrow Z\) be the inclusion maps. By the definition of \(\tps\) in \eqref{e def of Psi tilde},
\begin{equation}\label{e res of Psi}
\Psi\big|_X=\Ph_1 \quad \text{  and  }\quad \Ps\big|_Y=\Ph_2 .
\end{equation}
Thus, recalling the notation from Section \ref{s2.2} that \(H^*(\zn,\bbz_2)=\bbz_2[\overline{\la}]\),
\begin{equation}\label{e pullback by iota}
\Ps^*\overline{\la}\;\big|_X=\Ph_1^*\overline{\la}\quad \text{  and  }\quad \Ps^*\overline{\la}\;\big|_Y=\Ph_2^*\overline{\la}.
\end{equation}
In particular, this implies $\Ps^*\overline{\la}$ is non-zero and hence $\Ps^*\overline{\la}=\la$. \tf to prove that \(\Ps:Z \ra \zn\) is a \((k+m+1)\)-sweepout, it is enough to find \(c \in H_{k+m+1}(Z,\bbz_2)\) \st \(\la^{k+m+1}.c=1.\)

\Sn \(\Ph_1\) is a \(k\)-sweepout, there exists \(\al\in H_k(X,\bbz_2)\) \st
\begin{equation}\label{e lambda.alpha}
\left(\Ph_1^*\overline{\la}\right)^k.\al=1.
\end{equation}
Let \(\al=\left[\sum_{i=1}^I\al_i\right],\) where each \(\al_i:\De^k\ra X\) is a singular \(k\)-simplex and
\begin{equation}\label{e del alplha}
\del \Big(\sum_i\al_i\Big)=0.
\end{equation}
Similarly, as \(\Ph_2\) is an \(m\)-sweepout, there exists \(\be\in H_m(Y,\bbz_2)\) \st
\begin{equation}\label{e lambda.beta}
\left(\Ph_2^*\overline{\la}\right)^m.\be=1.
\end{equation}
Let \(\be=\left[\sum_{j=1}^J\be_j\right],\) where each \(\be_j:\De^m\ra Y\) is a singular \(m\)-simplex and
\begin{equation}\label{e del beta}
\del \Big(\sum_j\be_j\Big)=0.
\end{equation}

Let \(\tau:C_k(X,\bbz_2)\ra C_k(\tx, \bbz_2)\) be the transfer homomorphism defined as follows (see \cite{hat}*{Proof of Proposition 2B.6}). If \(\et:\De^k \ra X\) is a \cts map, \(\et\) has exactly two lifts \(\tilde{\et}_1,\tilde{\et}_2:\De^k\ra \tx\). One defines \(\ta(\et)=\tilde{\et}_1+\tilde{\et}_2\). Let \(\tilde{\al}_{i,1},\;\tilde{\al}_{i,2}:\De^k \ra \tx\) be the lifts of \(\al_i\) so that \(\ta(\al_i)=\tilde{\al}_{i,1}+\tilde{\al}_{i,2}\). We also choose \(\tilde{\be}_j:\De^m\ra \ty\), which is a lift of \(\be_j\).
\begin{clm}\label{clm del c}
\begin{equation}\label{e del c 0}
\del \rh_{\#}\Big(\sum_{i,j}\ta(\al_i)*\tilde{\be}_j\Big)=0\; \text{ in }C_{k+m}(Z,\bbz_2).
\end{equation}
\end{clm}
\begin{proof}
By \eqref{e join of sing simpl},
\begin{equation}\label{e del c 1}
\del\Big(\sum_{i,j}\ta(\al_i)*\tilde{\be}_j\Big)=\sum_{i,j}\del\ta(\al_i)*\tilde{\be}_j+\sum_{i,j}\ta(\al_i)*\del\tilde{\be}_j.
\end{equation}
Since \(\ta\) is a chain map (i.e. \(\ta\) commutes with \(\del\)), \eqref{e del alplha} implies
\begin{equation}\label{e del c 2}
\sum_{i}\del\ta(\al_i)=0.
\end{equation}
Further, by \eqref{e del beta},
\[\rh_{\#}\Big(\sum_j\del \tilde{\be}_j\Big)=0.\]
\hn there exist singular \((m-1)\)-simplices \(\tht_{r,1},\;\tht_{r,2}:\De^{m-1}\ra \ty\) \st \(\tht_{r,2}=T_2\circ \tht_{r,1}\) and
\[\sum_j\del \tilde{\be}_j=\sum_{r=1}^R\tht_{r,1}+\tht_{r,2}.\]
Thus
\begin{equation}\label{e del c 3}
\sum_{i,j}\ta(\al_i)*\del\tilde{\be}_j=\sum_{i,r}\tilde{\al}_{i,1}*\tht_{r,1}+\tilde{\al}_{i,2}*\tht_{r,1}+\tilde{\al}_{i,1}*\tht_{r,2}+\tilde{\al}_{i,2}*\tht_{r,2}.
\end{equation}
However,
\begin{equation}\label{e del c 4}
\rho\circ(\tilde{\al}_{i,1}*\tht_{r,1})=\rh\circ(\tilde{\al}_{i,2}*\tht_{r,2}) \quad \text{ and }\quad \rho\circ(\tilde{\al}_{i,2}*\tht_{r,1})=\rh\circ(\tilde{\al}_{i,1}*\tht_{r,2}).
\end{equation}
\tf combining \eqref{e del c 1}, \eqref{e del c 2}, \eqref{e del c 3} and \eqref{e del c 4}, we obtain \eqref{e del c 0}.
\end{proof}

Claim \ref{clm del c} implies that 
\begin{equation}\label{e def c}
c=\bigg[\rh_{\#}\Big(\sum_{i,j}\ta(\al_i)*\tilde{\be}_j\Big)\bigg] 
\end{equation}
is a well-defined homology class in \(H_{k+m+1}(Z,\bbz_2).\)
\begin{clm}
\(\la^{k+m+1}.c=1\), where \(c\) is as defined in \eqref{e def c}.
\end{clm}
\begin{proof}
Let us choose \(h:C_{1}(Z,\bbz_2)\ra \bbz_2\) \st \(\la=[h]\). \sps \(\{\nu_1,\nu_2,\dots,\nu_{k+m+2}\}\) are the vertices of \(\De^{k+m+1}\). We note that 
\begin{equation}\label{e.al}
\rh_{\#}(\tilde{\al}_{i,1}*\tilde{\be}_j)\big|_{[\nu_1,\dots,\nu_{k+1}]}=\iota_1\circ\al_i=\rh_{\#}(\tilde{\al}_{i,2}*\tilde{\be}_j)\big|_{[\nu_1,\dots,\nu_{k+1}]}.
\end{equation}
Similarly,
\begin{equation}\label{e.be}
\rh_{\#}(\tilde{\al}_{i,1}*\tilde{\be}_j)\big|_{[\nu_{k+2},\dots,\nu_{k+m+2}]}=\iota_2\circ\be_j=\rh_{\#}(\tilde{\al}_{i,2}*\tilde{\be}_j)\big|_{[\nu_{k+2},\dots,\nu_{k+m+2}]}.
\end{equation}
Moreover,
\begin{align*}
&(\tilde{\al}_{i,1}*\tilde{\be}_j)\big|_{[\nu_{k+1},\nu_{k+2}]}\text{ is a curve joining }\tilde{\al}_{i,1}(\nu_{k+1}) \text{  and  }\tilde{\be}_j(\nu_{k+2});\\
&(\tilde{\al}_{i,2}*\tilde{\be}_j)\big|_{[\nu_{k+1},\nu_{k+2}]}\text{ is a curve joining }\tilde{\al}_{i,2}(\nu_{k+1}) \text{  and  }\tilde{\be}_j(\nu_{k+2}).
\end{align*}
\Sn \(\tilde{\al}_{i,2}=T_1\circ\tilde{\al}_{i,1},\)
\[\rh_{\#}(\tilde{\al}_{i,1}*\tilde{\be}_j)\big|_{[\nu_{k+1},\nu_{k+2}]}+\rh_{\#}(\tilde{\al}_{i,2}*\tilde{\be}_j)\big|_{[\nu_{k+1},\nu_{k+2}]}\]
is a non-contractible loop in \(Z\). \hn using \eqref{e.al}, \eqref{e.be} and \eqref{e labmda.c}, we obtain
\begin{align}
&h^{k+m+1}\left(\rh_{\#}(\tilde{\al}_{i,1}*\tilde{\be}_j)+\rh_{\#}(\tilde{\al}_{i,2}*\tilde{\be}_j)\right) \nonumber\\
&=h^k\left(\iota_1\circ\al_i\right)h\left(\rh_{\#}(\tilde{\al}_{i,1}*\tilde{\be}_j)\big|_{[\nu_{k+1},\nu_{k+2}]}+\rh_{\#}(\tilde{\al}_{i,2}*\tilde{\be}_j)\big|_{[\nu_{k+1},\nu_{k+2}]}\right)h^m\left(\iota_2\circ\be_j\right)\nonumber\\
&=h^k(\iota_1\circ\al_i)h^m(\iota_2\circ\be_j).\label{e h^k.h^m}
\end{align}
\eqref{e pullback by iota}, \eqref{e lambda.alpha}, \eqref{e lambda.beta} and \eqref{e h^k.h^m} imply that \(\la^{k+m+1}.c=1\).
\end{proof}

{\bf Part 4.} \Sn \(\Ps:Z \ra \zn\) is a \((k+m+1)\)-sweepout with no concentration of mass,
\begin{equation}\label{e lower bd of mass of Psi}
\sup_{z\in Z}\{\bM(\Ps(z))\}\geq \om_{k+m+1}.
\end{equation}
The following proposition essentially follows from the proof of Proposition \ref{pr cac set}.
\begin{pro}\label{p mass of Psi tilde}
	For all \((x,y,t)\in \tx*\ty,\)
	\[\bM(\del \tps(x,y,t))\leq \bM(\Ph_1(x))+\bM(\Ph_2(y))+\bM(\del \La(t)).\]
\end{pro}
\begin{proof}
By \cite{MPPP}*{Proposition 1.4}, there exist \(\{u_i\}_{i=1}^{\infty},\;\{v_i\}_{i=1}^{\infty},\; \{w_i\}_{i=1}^{\infty} \subset C^{\infty}(M)\) \st \(0\leq u_i,v_i,w_i\leq 1\) for all \(i\);
\begin{align*}
& u_i \ra \ch_{\tp_1(x)} \text{ in } L^1(M) \text{ and pointwise a.e.};\quad \bM(\Ph_1(x))=\lim_{i \ra \infty}\int_M |\na u_i|\; d\cH^{n+1};\\
& v_i \ra \ch_{\tp_2(y)} \text{ in } L^1(M) \text{ and pointwise a.e.};\quad \bM(\Ph_2(y))=\lim_{i \ra \infty}\int_M |\na v_i|\; d\cH^{n+1};\\
& w_i \ra \ch_{\La(t)} \text{ in } L^1(M) \text{ and pointwise a.e.};\quad \bM(\del\La(t))=\lim_{i \ra \infty}\int_M |\na w_i|\; d\cH^{n+1}.
\end{align*}
By the dominated convergence theorem,
\[u_iw_i+v_i(1-w_i)\ra \ch_{\tps(x,y,t)} \;\text{ in } L^1(M).\]
\tf
\begin{align*}
\bM(\del \tps(x,y,t))&\leq \liminf_{i \ra \infty}\int_M\Big(\big|\na (u_iw_i)+\na (v_i(1-w_i))\big|\Big)\;d\cH^{n+1}\\
&\leq \liminf_{i \ra \infty}\int_M\Big(\md{\na u_i}+\md{\na v_i}+\md{\na w_i}\Big)\;d\cH^{n+1}\\
&=\bM(\Ph_1(x))+\bM(\Ph_2(y))+\bM(\del\La(t)).
\end{align*}
\end{proof}
Combining Proposition \ref{p mass of Psi tilde}, \eqref{e mass of Phi_1}, \eqref{e mass of Phi_2}, \eqref{e mass of del Laambda} and \eqref{e lower bd of mass of Psi}, we obtain
\[\om_{k+m+1}\leq \om_k+\om_m+W+3\de.\]
\Sn \(\de>0\) is arbitrary, this finishes the proof of Theorem \ref{thm 1.1}.

\section{Proof of Theorem \ref{thm 1.2}}\newcommand{\ce}{c_{\ve}}\newcommand{\tlh}{\tilde{h}}\newcommand{\tz}{\tilde{\ze}}
If \(k=m=0\), \eqref{eq1.2} reduces to \(c_{\ve}(1)\leq \ga_{\ve}\), which holds by the definitions of \(c_{\ve}(1)\) and \(\ga_{\ve}\) (see \cite{GG1}*{Remark 3.6}). So let us assume that \(k\geq 1.\) We fix \(\de>0\). There exists a connected simplicial complex \(\tx\in \cC_k\) and \(\tlh_1\in \Ga(\tx)\) \st 
\begin{equation}\label{e.h1.tilde.energy}
\sup_{x\in\tx}E_{\ve}\big(\tlh_1(x)\big)\leq \ce(k)+\de.
\end{equation}
Similarly, for \(m\geq 1\), there exists a connected simplicial complex \(\ty\in \cC_m\) and \(\tlh_2\in \Ga(\ty)\) \st 
\begin{equation}\label{e.h2.tilde.energy}
\sup_{y\in\ty}E_{\ve}\big(\tlh_2(y)\big)\leq \ce(m)+\de.
\end{equation}
If \(m=0\), we choose \(\ty=\{-1,1\}\) and \(\tlh_2(\pm 1)=\pm\mathbf{1}\). Let \(w:[0,1]\ra H^1(M)\) be \st \(w(0)=\mathbf{1}\), \(w(1)=-\mathbf{1}\) and
\begin{equation}\label{e.w.energy}
\sup_{t\in[0,1]}E_{\ve}(w(t))\leq \ga_{\ve}+\de.
\end{equation}
Suppose \(\cT:H^1(M)\ra H^1(M)\) is the following truncation map (see \cite{Guaraco}*{Section 4}).
\[\cT(u)=\min\{\max\{u,-1\},1\}.\]
Then \(\cT\) is continuous, \(-1\leq\cT(u)\leq 1\), \(\cT(-u)=-\cT(u)\) and
\[E_{\ve}\left(\cT(u)\right)\leq E_{\ve}(u).\]
\tf without loss of generality, we can assume that
\begin{equation}\label{e.trunc}
-1\leq \tlh_1(x),\;\tlh_2(y),\;w(t)\leq 1
\end{equation}
for all \(x\in \tx\), \(y \in \ty\) and \(t \in [0,1]\).

For \(a\in \bbr\), let \(a^+=\max\{a,0\}\); \(a^-=\min\{a,0\}\). The following maps \(\ph,\ps,\tht:H^1(M)\times H^1(M)\times H^1(M)\ra H^1(M)\) were defined in \cite{d2}.
\begin{align*}
& \ph(u_1,u_2,v)=\min\{\max\{u_1,-v\},\max\{u_2,v\}\};\\
& \ps(u_1,u_2,v)=\max\{\min\{u_1,v\},\min\{u_2,-v\}\};\\
& \tht(u_1,u_2,v)=\ph(u_1,u_2,v)^++\ps(u_1,u_2,v)^-.
\end{align*}
\begin{pro}[\cite{d2}*{Proposition 3.12.}]\label{p.tht}
The map \(\tht\), defined above, has the following properties.
\begin{itemize}
	\item[(i)] \(\tht\) is continuous.
	\item[(ii)] \(\tht(-u_1,-u_2,v)=-\tht(u_1,u_2,v).\)
	\item[(iii)] If \(-1\leq u_1,u_2\leq1\), \(\tht(u_1,u_2,\mathbf{1})=u_1\) and \(\tht(u_1,u_2,-\mathbf{1})=u_2\).
	\item[(iv)] \(E_{\ve}(\tht(u_1,u_2,v))\leq E_{\ve}(u_1)+E_{\ve}(u_2)+E_{\ve}(v)\).
\end{itemize}	
\end{pro}\newcommand{\tq}{\tilde{q}}

We define \(\tq:\tx*\ty\ra H^1(M)\) as follows.
\begin{equation}\label{e.def.zeta.tilde}
\tq(x,y,t)=\tht\left(\tlh_1(x),\tlh_2(y),w(t)\right).
\end{equation}
By \eqref{e.trunc} and Proposition \ref{p.tht}, items (i), (iii), \(\tq\) is a well-defined, \cts map and 
\begin{equation}
\tq\big|_{\tx}=\tlh_1;\quad\tq\big|_{\ty}=\tlh_2.
\end{equation}
(We recall from Part \(3\) of the proof of Theorem \ref{thm 1.1} that \(\tx\) and \(\ty\) can be naturally identified as subspaces of \(\tx*\ty\).)
\begin{clm}\label{c.zeta.sweepout}
\(\tilde{Z}=\tx*\ty\in \cC_{k+m+1}\) and there exists \(\tz\in \Ga(\zt)\) \st
\begin{equation}
E_{\ve}(\tz(z))\leq E_{\ve}(\tq(z))+\de\quad\forall z\in \zt.\label{e.cl.energy.tz'}
\end{equation}
\end{clm}
\begin{proof}
Let \(T_1:\tx\ra \tx\) be \st the \(\bbz_2\) action on \(\tx\) is given by \(x\mapsto T_1(x)\) and \(T_2:\ty\ra\ty\) be \st the \(\bbz_2\) action on \(\ty\) is given by \(y \mapsto T_2(y)\). There exist simplicial complex structures on \(\tx\) and \(\ty\) \st the maps \(T_1\) and \(T_2\) are simplicial homeomorphisms. Let \(\{\De_{1,i}:i=1,\dots,I\}\) and \(\{\De_{2,j}:j=1,\dots,J\}\) be the set of simplices in \(\tx\) and \(\ty\). As discussed in Section \ref{s2.4}, one can put a canonical simplicial complex structure on \(\zt=\tx*\ty\), whose simplices are \(\{\De_{1,i},\;\De_{2,j},\;\De_{1,i}*\De_{2,j}:i=1,\dots,I;\;j=1,\dots,J\}\). There is a free, simplicial \(\bbz_2\) action on \(\zt\) given by 
\[z=(x,y,t)\mapsto T(z)=(T_1(x),T_2(y),t).\]
By Proposition \ref{p.tht}(ii), \(\tq(T(z))=-\tq(z)\) for all \(z\in \zt\). There might exist \(z\in \zt\) \st \(\tq(z)=0\in H^1(M)\). However, as we discuss below, one can perturb \(\tq\) and obtain a new map \(\tz:\zt \ra H^1(M)\setminus\{0\}\), which is \(\bbz_2\)-equivariant and satisfies
\begin{align}
& E_{\ve}(\tz(z))\leq E_{\ve}(\tq(z))+\de\quad\forall z\in \zt;\label{e.energy.tz'}\\
& \tz\big|_{\tx}=\tq\big|_{\tx}=\tlh_1;\quad\tz\big|_{\ty}=\tq\big|_{\ty}=\tlh_2. \label{e.zeta.res}
\end{align}

Let \(\{e_s,f_s:s=1,\dots,S\}\) be the set of all simplices in \(\zt\), indexed in such a way that \(T(e_s)=f_s\) and \(s_1\leq s_2\) implies \(\dim(e_{s_1})\leq \dim(e_{s_2})\). We will define \(\tz\) inductively on the simplices of \(\zt\). By abuse of notation, a simplex of \(\zt\) will be identified with its support. For \(R>0\), let 
\[B_R=\{u \in H^1(M):\|u\|_{H^1(M)}<R\}.\]
\(B_R\setminus\{0\}\), being homeomorphic to \(H^1(M)\setminus\{0\}\), is contractible. Let \(r>0\) be \st
\begin{itemize}
	\item \(B_r\) is disjoint from \(\tlh_1(\tx)\) and \(\tlh_2(\ty)\) (we note that \(\tlh_1(\tx)\) and \(\tlh_2(\ty)\) are compact subsets of \(H^1(M)\), which are disjoint from \(\{0\}\));
	\item \(\md{E_{\ve}(u)-E_{\ve}(0)}<\de/2\) for all \(u \in B_r\).
\end{itemize}
Suppose
\[\zt_0=\{z\in\zt:\tq(z)=0\in H^1(M)\}.\]
Using barycentric subdivision, without loss of generality one can assume that 
\begin{equation}\label{e.al.assumption}
\text{if } \al \text{ is a simplex in } \zt \text{ and } \al\cap\zt_0\neq \emptyset, \text{ then } \tq(\al) \subset B_r. 
\end{equation}

If \(s\) is \st \(\dim(e_s)=0\), we set 
\begin{equation}
\tz(e_s)=\begin{cases}
\tq(e_s)& \text{ if }\tq(e_s)\neq 0;\\
\text{an arbitrary point in }B_r\setminus\{0\} & \text{ if }\tq(e_s)= 0;
\end{cases}
\end{equation}
and \(\tz(f_s)=-\tz(e_s)\). Let us assume that we have defined
\[\tz:\bigcup_{s\leq l-1}(e_s\cup f_s) \ra H^1(M)\setminus\{0\}\]
in such a way that
\begin{itemize}
	\item[(i)] \(\tz(T(z))=-\tz(z)\) for all \(z\in \bigcup_{s\leq l-1}(e_s\cup f_s)\);
	\item[(ii)] \(\tz\big|_{e_s\cup f_s}=\tq\big|_{e_s\cup f_s}\) if \(s\leq l-1\) and \(e_s\cap\zt_0= \emptyset\) (which is equivalent to \(f_s\cap\zt_0= \emptyset\));
	\item[(iii)] \((\tz(e_s)\cup\tz(f_s))\subset B_r\setminus\{0\}\) if \(s\leq l-1\) and \(e_s\cap\zt_0\neq \emptyset\) (which is equivalent to \(f_s\cap\zt_0\neq \emptyset\)).
\end{itemize}
We define \(\tz:(e_l\cup f_l) \ra H^1(M)\setminus\{0\}\) as follows. By the induction hypothesis, \(\tz\) has already been defined on \((\del e_l \cup \del f_l)\). If $e_l\cap\zt_0= \emptyset=f_l\cap\zt_0$, we set
\[\tz\big|_{e_l \cup f_l}=\tq\big|_{e_l \cup f_l}.\] 
Otherwise, \(\tz\) is defined on \(e_l\) in such a way that the definition of \(\tz\) on \(\del e_l\) matches with its previous definition and \(\tz(e_l) \subset B_r\setminus\{0\}\). Such an extension is possible as \(B_r\setminus\{0\}\) is contractible. For \(z\in f_l\), \(\tz(z)\) is defined to be equal to \(-\tz(T(z))\). Thus we have defined \(\tilde{\ze}:\zt\ra H^1(M)\setminus \{0\},\) which is \(\bbz_2\)-equivariant. \eqref{e.energy.tz'} and \eqref{e.zeta.res} follow from the choice of \(r\) and \eqref{e.al.assumption}.

\sps \(Z\) and \(\sH(M)\) denote the quotients of \(\zt\) and \(H^1(M)\setminus \{0\}\), respectively, by \(\bbz_2\). Then \(\tz\) descends to a map \(\ze:Z\ra \sH(M)\). As discussed in Section \ref{s2.3}, \(H^*(\sH(M),\bbz_2)=\bbz_2[\xi]\). Furthermore, \(\zt\in \cC_{k+m+1}\) if and only if 
\begin{equation}\label{e.pullbak.zeta}
\ze^*(\xi^{k+m+1})\neq 0 \in H^{k+m+1}(Z,\bbz_2).
\end{equation}
However, \sn \(\tx\in \cC_k\), if \(X\) denotes the quotient of \(\tx\) by \(\bbz_2\), \(\tlh_1\) descends to a map \(h_1:X\ra \sH(M)\) \st
\begin{equation}\label{e.pullback.h1}
h_1^*(\xi^k)\neq 0 \in H^k(X,\bbz_2).
\end{equation}
Similarly, as \(\ty\in\cC_m\), if \(Y\) denotes the quotient of \(\ty\) by \(\bbz_2\), \(\tlh_2\) descends to a map \(h_2:Y\ra \sH(M)\) \st
\begin{equation}\label{e.pullback.h2}
h_2^*(\xi^m)\neq 0 \in H^m(Y,\bbz_2).
\end{equation}
One can prove \eqref{e.pullbak.zeta} by using \eqref{e.pullback.h1}, \eqref{e.pullback.h2}, \eqref{e.zeta.res} and an argument similar to Part \(3\) of the proof of Theorem \ref{thm 1.1}.
\end{proof}

By the above Claim \ref{c.zeta.sweepout},
\begin{equation}\label{e.k+m+1}
\sup_{z\in \tx*\ty}E_{\ve}\big(\tz(z)\big)\geq \ce(k+m+1).
\end{equation}
On the other hand, by \eqref{e.cl.energy.tz'} and Proposition \ref{p.tht}(iv), for all \((x,y,t)\in \tx*\ty\),
\begin{equation}\label{e.k.m.1}
E_{\ve}\big(\tz(x,y,t)\big)\leq E_{\ve}\big(\tq(x,y,t)\big)+\de\leq E_{\ve}\big(\tlh_1(x)\big)+E_{\ve}\big(\tlh_2(y)\big)+E_{\ve}(w(t))+\de.
\end{equation}
Combining \eqref{e.k+m+1}, \eqref{e.k.m.1}, \eqref{e.h1.tilde.energy}, \eqref{e.h2.tilde.energy} and \eqref{e.w.energy}, we obtain
\[\ce(k+m+1)\leq \ce(k)+\ce(m)+\ga_{\ve}+4\de.\]
\Sn \(\de>0\) is arbitrary, this finishes the proof of Theorem \ref{thm 1.2}.

\medskip
\nocite{*}
\bibliographystyle{amsalpha}
\bibliography{widthsub}

\end{document}